\documentclass[a4paper]{amsart}

\usepackage{amsmath,amssymb}
\usepackage{amsthm}
\pdfoutput=1

\usepackage{graphicx}
\usepackage{bm}
\usepackage{tikz-cd}
\usepackage{mathtools}
\usepackage{comment}
\usepackage{bigdelim, multirow}
\usepackage{latexsym}
\usepackage{stmaryrd}
\usepackage{braket}
\usepackage{color}
\usepackage{setspace}
\usepackage{comment}
\usepackage{url}
\usepackage{setspace}
\newcommand{\Z}{\mathbb{Z}}

\newcommand{\R}{\mathbb{R}}
\newcommand{\C}{\mathbb{C}}

\newcommand{\F}{\mathbb{F}}

\newcommand{\gerg}{\mathfrak{g}}
\newcommand{\gergl}{\mathfrak{gl}}

\newcommand{\calB}{\mathcal{B}}

\newcommand{\Cstar}{\C^*}

\newcommand\Hom{\mathop{\mathrm{Hom}}\nolimits}

\newcommand\Ann{\mathop{\mathrm{Ann}}\nolimits}

\newcommand\Ker{\mathop{\mathrm{Ker}}\nolimits}

\newcommand\rank{\mathop{\mathrm{rank}}\nolimits}

\def\fr<#1/#2>{\frac{#1} {#2}}

\newcommand{\tr}{\operatorname{tr}}
\newcommand{\Sym}{\operatorname{Sym}}

\newtheorem{thm}{Dont use this}[section]
\newtheorem{theorem}[thm]{Theorem}
\newtheorem{proposition}[thm]{Proposition}
\newtheorem{corollary}[thm]{Corollary}
\newtheorem{lemma}[thm]{Lemma}

\newtheorem*{definition*}{Definition}

\newtheorem{definition and lemma}[thm]{Definition $\&$ Lemma}

\newtheorem*{remark*}{Remark}
\theoremstyle{definition}
\newtheorem{definition}[thm]{Definition}
\newtheorem{example}[thm]{Example}

\newtheorem{obs}[thm]{Key Observation}
\newtheorem{remark}[thm]{Remark}

\theoremstyle{definition}

\title{%
Strict log-concavity of the Kirchhoff polynomial and its applications to the strong Lefschetz property
}
\author[T. Nagaoka]{Takahiro Nagaoka}
\address[Takahiro Nagaoka]{Department of Mathematics,
	Graduate School of Science,
	Kyoto University,
	Kyoto, 606-8522, Japan}
\email{tnagaoka@math.kyoto-u.ac.jp}
\author[A. Yazawa]{Akiko Yazawa}
\address[Akiko Yazawa]{Department of Science and Technology,
	Graduate School of Medicine, Science and Technology,
	Shinshu University,
	Matsumoto, Nagano, 390-8621, Japan}
\email{yazawa@math.shinshu-u.ac.jp}
\date{}

\begin{document}
\pagestyle{plain}
\begin{abstract}
Anari, Gharan, and Vinzant proved (complete) log-concavity of the basis generating functions for all matroids. 
From the viewpoint of combinatorial Hodge theory, 
it is natural to ask whether these functions are ``strictly" log-concave for simple matroids. 
In this paper, we show this strictness for simple graphic matroids, 
that is, we show that Kirchhoff polynomials of simple graphs are strictly log-concave. 
Our key observation is that the Kirchhoff polynomial {of a complete graph} can be seen as the (irreducible) relative invariant of a certain prehomogeneous vector space, 
which may be independently interesting in its own right. 
Furthermore, we prove that for any $a_i\in\R_{>0}$, $a_1x_1+\cdots+a_nx_n\in R^1_{M}$ satisfies the strong Lefschetz property 
(moreover, Hodge--Riemann bilinear relation) at degree one of the Artinian Gorenstein algebra $R^*_M$ associated to a graphic matroid $M$, 
which is defined by Maeno and Numata for all matroids. 
\end{abstract} 
\maketitle
\setcounter{tocdepth}{1}
\tableofcontents

\section{Introduction}\label{sec:Intro}
The Kirchhoff polynomial $F_\Gamma$ of a graph $\Gamma=(V, E)$ is a multi-affine homogeneous polynomial of degree $r$ in $n$ variables, where $n=|E|$ and $r=|V|-1$. 
Such polynomials are important in several areas of study such as network theory and physics (where these polynomials are related to Feynman diagrams). 
Also, the Kirchhoff polynomial can be seen as a special case of the basis generating function $F_M$ for a graphic matroid $M$. 
The properties of the basis generating function, for example the half-plane property, have been extensively studied in \cite{MR2037144}. 
Recently, in \cite{AGV2018}, Anari, Gharan, and Vinzant showed that $F_M$ satisfies log-concavity (more precisely, complete log-concavity) on $\R_{\geq0}^n$. 
In other words, they show that $\log F_M$ is concave on $\R^n_{\geq0}$, 
that is the Hessian matrix $H_{F_M}$ and the gradient vector $\nabla F_M$ of $F_M$ satisfy 
\begin{equation*}\tag{*}
\left.\left(-F_MH_{F_M}+\nabla F_M(\nabla F_M)^T\right)\right|_{\bm{x}=\bm{a}} \ \text{is positive semi-definite} 
\end{equation*} 
for any $\bm{a}\in\R_{\geq0}^n$. 
Their proof is based on the combinatorial Hodge theory developed in \cite{MR3862944} and \cite{MR3733101}. 
As will be explained later in this introduction, if (*) is ``positive definite", then a certain Hodge--Riemann bilinear form is non-degenerate. 
Thus, from the view point of the combinatorial Hodge theory, it is important to know whether or not the basis generating function is strictly log-concave on $(\R_{>0})^n$. 
In our main theorem, we claim that for simple graphs, this stands true for the following statements. 
\begin{theorem}[cf.\ Theorem \ref{mainresult}]\hspace{2pt}\\
For any simple graph $\Gamma$ with $r+1$ vertices and $n$ edges, the Kirchhoff polynomial $F_\Gamma$ is strictly log-concave on $(\R_{>0})^n$. In other words, for any $\bm{a}\in(\R_{>0})^n$, $\log F_\Gamma$ is strictly concave at $\bm{a}$, that is,
\[(-F_\Gamma H_{F_\Gamma}+\nabla F_{\Gamma}(\nabla F_\Gamma)^T)|_{\bm{x}=\bm{a}} \ \text{is positive definite}.\]
In particular, $H_{F_\Gamma}|_{\bm{x}=\bm{a}}$ is non-degenerate, with $n-1$ negative eigenvalues and one positive eigenvalue. 
Thus, $(-1)^{n-1}(\det H_{F_\Gamma})|_{\bm{x}=\bm{a}}>0$.
\end{theorem}
The proof of our main theorem comprises two steps. 
First, we reduce our claim to the following  determinantal identity of the Hessian of the Kirchhoff polynomial $F_{K_{r+1}}$ of complete graphs $K_{r+1}$ (cf.\ Theorem \ref{maintheorem2}).  
\begin{equation*}
\det H_{F_{K_{r+1}}}=(-1)^{N-1}c_r(F_{K_{r+1}})^{N-r-1},
\end{equation*}
where $c_r>0$ is a constant, and $N:=\binom{r+1}{2}$. 
Second, we show the above equality not by directly computing but rather by identifying $F_{K_{r+1}}$ 
with the unique irreducible polynomial associated to a special $GL_r(\C)$ representation 
or the so-called prehomogeneous vector space. 
Then, based on the general theory of prehomogeneous vector spaces \cite{MR0430336}, 
the Hessian $\det H_F$ of the relative invariant $F$ is also a relative invariant of the same representation.  
Hence we have 
\[\exists c \in \C \ \text{such that} \ \det H_F=cF^m\]
by the uniqueness of the relative invariant. 
We believe that this method may be useful for proving some (conjectural) Hessian identity in general. 

Recent studies by Br\"{a}nd\'{e}n and Huh revealed that
the Hessian of a nonzero Lorentzian polynomial has exactly one positive eigenvalue at any point on the positive orthant (see \cite[Sections 5 and 7]{BH2019}). 

\medskip

In Section \ref{sec:app1}, we give some applications of the main theorem to the strong Lefschetz property of the graded Artinian Gorenstein algebra $R_\Gamma^*=\oplus_{\ell=0}^r{R_\Gamma^\ell}=\R[x_1, \ldots, x_n]/\Ann(F_\Gamma)$ associated to any simple graph $\Gamma$ (see Definition \ref{def:algebra}). 
This algebra is defined for any matroid $M$ by Maeno and Numata who proved that this algebra has the strong Lefschetz property at all degrees when $M$ is the projective space $M(q, n)$ over a finite field (they denote $R_M^*$ by $A_M$) in \cite{MR3566530}.  
In an extended abstract \cite{MR2985388} of the paper \cite{MR3566530}, 
Maeno and Numata also conjectured that $R^*_M$ has the strong Lefschetz property for any matroid $M$.  
As an application of our main theorem, we prove that this conjecture at degree one when $M$ is a graphic matroid, with the following. 
\begin{theorem}[cf.\ Theorem \ref{cor:SLPgraph}]\hspace{2pt}\\
For any simple graph $\Gamma$ with $r+1$ vertices and $n$ edges, 
and any $\bm{a}=(a_1, \ldots, a_n)\in(\R_{>0})^n$, $L_{\bm{a}}:=a_1x_1+\cdots+a_nx_n\in R_{\Gamma}^1$ satisfies the strong Lefschetz property at degree one, 
that is, the multiplication map  
\[\times L_{\bm{a}}^{r-2} : R_\Gamma^1 \to R_\Gamma^{r-1}\]
is an isomorphism. 
\end{theorem}
Since the Hodge--Riemann bilinear form (see Definition \ref{HRR}) of $R^1_\Gamma$ is given by the Hessian $H_{F_\Gamma}$, we have the following stronger application. 
\begin{theorem}[cf.\ Theorem \ref{cor:HRRgraph}]\hspace{2pt}\\
In the above setting, for any $\bm{a}\in(\R_{>0})^n$, the Hodge--Riemann bilinear form
\[{Q_{L_{\bm{a}}}^1} : R_\Gamma^1\times R_\Gamma^1 \to \R, \ \ \ (\xi_1, \xi_2) \to {[\xi_1 L_{\bm{a}}^{r-2}\xi_2]}\]
is non-degenerate, where $[-] :R_\Gamma^r\xrightarrow{\sim} \R$ is the isomorphism as 
\begin{align*}
P \mapsto P\left(\fr<\partial/\partial x_1>, \ldots, \fr<\partial/\partial x_n>\right)F_\Gamma. 
\end{align*}
Moreover, ${Q_{L_{\bm{a}}}^1}$ has $n-1$ negative eigenvalues and one positive eigenvalue. 
\end{theorem}
As we note in Remark \ref{rem:HW}, 
our ${Q_{L_{\bm{a}}}^1}$ is the same as the Hodge--Riemann bilinear form on the degree one part of another algebra $B^*(M)$ when $M=\Gamma$, which is studied in \cite{MR3733101}. 
In \cite[Remark 15]{MR3733101}, Huh and Wang considered the Hodge--Riemann bilinear form on $B^1(M)$ for a general simple matroid. 
Our corollary then implies the same conclusion for $B^*(M)$ at degree one as the above theorem (in general, there exists a natural surjection $B^*(M)\twoheadrightarrow R^*_M$).

This paper is organized as follows. In Section \ref{sec:Homogeneous polynomials}, 
we study the properties of homogeneous polynomials in terms of their Hessian and log-concavity.  
In particular, we collect some propositions on prehomogeneous vector spaces in Subsection \ref{subsec:Prehomogeneous vector spaces}. 
In Section \ref{sec:Matroid}, 
we see several definitions and propositions for matroids. 
In Section \ref{sec:Main result}, 
we define the Kirchhoff polynomials of simple graphs, and then prove our main result. 
In the last half of this section, 
we see that the connection between the Kirchhoff polynomials of complete graphs and certain prehomogeneous vector spaces. 
Finally, in Section \ref{sec:Applications}, 
we conclude that our main result gives applications to algebras associated to graphic matroids. 
\medskip

\noindent\textbf{Acknowledgements.}
The authors wish to express their gratitude to Yasuhide Numata, who is the second author's supervisor, for suggesting this problem. They are also grateful to Hiroyuki Ochiai for offering general facts on Proposition \ref{identity1} and for helpful comments on Remark \ref{rem:completegraph}, as well as June Huh and Satoshi Murai for their helpful comments concerning Remark \ref{rem:HW}. 

\section{Homogeneous polynomials}\label{sec:Homogeneous polynomials}

Let us consider a homogeneous polynomial $F$ of degree $r$ in $n$ variables with real coefficients, where $r\geq2$. 
For $F$, we define the Hessian matrix $H_{F}$ and the gradient vector $\nabla F$ by
\begin{align*}
H_F=\left(\fr<\partial^2 F/\partial x_i\partial x_j>\right)_{1\leq i, j\leq n}, &&
\nabla F =\left(\fr<\partial F/\partial x_1>, \ldots, \fr<\partial F/\partial x_n>\right). 
\end{align*}
We call $\det H_{F}$ the Hessian of $F$. 

In the first half of this section, we consider the Hessian of $F$.  
First, we see the following identity
\begin{align}\label{The Hessian of a homogeneous polynomial}
\det\left(-FH_{F}+s(\nabla F)^{T}\cdot \nabla F\right)=(-1)^{n-1}\frac{r}{r-1}\left(s-\frac{r-1}{r}\right)F^n\det H_F. 
\end{align}

Next, for a special polynomial $F$, we show the following identity 
\begin{align}\label{The Hessian in the case of the relative invariant}
\det H_{F}=c'F^{\frac{n(r-2)}{r}}, 
\end{align}
where $c'$ is non-zero. 

In the last half of this section, we consider the strict log-concavity of $F$. 

\subsection{The Hessians}
Here, we prove the identity \eqref{The Hessian of a homogeneous polynomial}. 
The set of all $m\times n$ matrices is denoted by $M_{m\times n}$. 
Moreover $I_{n}$ represents the $n \times n$ identity matrix. 

To prove \eqref{The Hessian of a homogeneous polynomial}, we prepared two lemmas. 

\begin{lemma}\label{lem:hom1}
For an $n \times n$ matrix $N$ of rank one, we have
\begin{align*}
\det(I_{n}-sN)=1-s(\tr N). 
\end{align*}
\end{lemma}

\begin{lemma}[Euler's identity]\label{lem:hom2}\hspace{2pt}\\
For a homogeneous polynomial $F$ of degree $r$ in $n$ variables, where $r\geq2$, we have 
\[r(r-1)F=\bm{x}^TH_F\bm{x},\]
\[(r-1)(\nabla F)^T=H_F\bm{x},\]
where $\bm{x}=(x_1, \ldots, x_n)^T$. 
\end{lemma}

Lemma \ref{lem:hom1} is straightforward. 
The proof of Lemma \ref{lem:hom2} is in \cite[Corollary 4.3.]{AGV2018}. 

\begin{proposition}\label{identity1}
For a homogeneous polynomial $F$ of degree $r$ in $n$ variables, where $r\geq2$, we have 
\begin{align*}
\det\left(-FH_{F}+s(\nabla F)^{T} \nabla F\right)=(-1)^{n-1}\frac{r}{r-1}\left(s-\frac{r-1}{r}\right)F^{n}\det H_F. 
\end{align*}
\end{proposition}

\begin{proof}
{Let $A$ be a $n\times n$ matrix and $\bm{v}$ a column vector of size $n$, where we consider every entries of $A$ and $\bm{v}$ as formal variables. }   
We set 
\begin{align*}
N=\frac{1}{\bm{v}^{T}A\bm{v}}(A\bm{v})(A\bm{v})^{T}A^{-1}.  
\end{align*}
In this case, we have $\rank N=1$ and $\tr N=1$. 
By Lemma \ref{lem:hom1}, 
\begin{align*}
\det\left(I_{n}-s\frac{1}{\bm{v}^{T}A\bm{v}}(A\bm{v})(A\bm{v})^{T}A^{-1}\right)=1-s. 
\end{align*}
If we multiply $\det A$ from the right, we obtain 
\begin{align*}
\det\left(A-s\frac{1}{\bm{v}^{T}A\bm{v}}(A\bm{v})(A\bm{v})^{T}\right)=(1-s)\det A. 
\end{align*}
If we multiply the left hand side by $(-\bm{v}^{T}A\bm{v})^n$, we obtain 
\begin{align*}\tag{*}
\det\left(-(\bm{v}^{T}A\bm{v})A+s(A\bm{v})(A\bm{v})^{T}\right)&=(1-s)(-\bm{v}^{T}A\bm{v})^{n}\det A \\
&=(-1)^{n-1}(s-1){(\bm{v}^{T}A\bm{v})}^{n}\det A. 
\end{align*}

Using Lemma \ref{lem:hom2}, we have the following identity 
\begin{align*}
-FH_F+s(\nabla F)^T\nabla F=\fr<1/r(r-1)>\left\{-(\bm{x}^TH_F\bm{x})H_F+\fr<sr/r-1>(H_F\bm{x})(H_F\bm{x})^T\right\}.
\end{align*}
Thus, applying (*) as $A=H_{F}$ and $\bm{v}=\bm{x}$, we obtain the desired equation. 
\end{proof}

By Proposition \ref{identity1}, we obtain Identity \eqref{The Hessian of a homogeneous polynomial}. 

\subsection{Prehomogeneous vector spaces}\label{subsec:Prehomogeneous vector spaces}
Here, we prove Identity \eqref{The Hessian in the case of the relative invariant} for the relative invariant of an irreducible prehomogeneous vector space (Corollary \ref{cor:reg}). 
To prove it, 
we introduce the notion of prehomogeneous vector spaces developed by Kimura and Sato \cite{MR0430336} and many authors. To be self-contained, we obtained certain useful propositions in \cite{MR0430336} and provided their proofs. 
Essentially, we followed \cite{MR0430336} while also using the notations mentioned in \cite{MR1944442}. 
\begin{definition}[Prehomogeneous vector space cf.\ {\cite[Definition 1 in Section 2 \& p.36]{MR0430336}}]\hspace{2pt}\\
Let $(G, \rho ,V)$ be a triplet of 
a connected linear algebraic group $G$,  
a finite dimensional vector space $V$, 
and a rational representation $\rho$ of $G$ on $V$, all defined over $\C$.  
We call $(G, \rho, V)$ a {\it prehomogeneous vector space} 
if there exists a proper algebraic $G$-invariant subset $S\subset V$ such that $V\setminus S$ is a single $G$-orbit. 
Then, we say that $S$ is the {\it singular set} of $(G, \rho, V)$. 
We say that $(G, \rho, V)$ is {\it irreducible} when $\rho$ is an irreducible representation. 
\end{definition}

\begin{definition}[Relative invariants cf.\ {\cite[Definition 2 in Section 4]{MR0430336}}]\hspace{2pt}\\
Let $(G, \rho ,V)$ be a prehomogeneous vector space.
A not identically zero rational function $F\in\C(V)$ is called a {\it relative invariant (with respect to $\chi$) of $(G, \rho ,V)$} 
if there exists a rational character $\chi \in \Hom(G, \Cstar)$ which satisfies the following:  
\[F(\rho(g)\bm{x})=\chi(g)F(\bm{x}) \ \ \ (g\in G, \bm{x}\in V).\]
In this case, we write $F\leftrightarrow \chi$. 
\end{definition}

Note that a relative invariant is a rational function on $V$, and not necessarily a polynomial on $V$. 
We define a subgroup $X_1(G)$ of $\Hom(G, \Cstar)$ by 
\[X_1(G):=\{\chi\in\Hom(G, \Cstar) \ | \ \exists F\in\C(V) \ \text{such that} \ F\leftrightarrow \chi\}.\]

\begin{remark}\label{rem:subgroup}
For any $\chi\in X_1(G)$, if $\rho(g_1)=\rho(g_2)$, then $\chi(g_1)=\chi(g_2)$. 
In particular, we can {consider} as $X_1(G)\subset\Hom(\rho(G), \Cstar)$ by the natural inclusion $\Hom(\rho(G), \Cstar)\hookrightarrow\Hom(G, \Cstar)$ induced from $G\twoheadrightarrow\rho(G)$.  
\end{remark}

\begin{proposition}[{\cite[Proposition 3 in Section 2]{MR0430336}}]\label{prop:const}\hspace{2pt}\\
Let $(G, \rho, V)$ be a prehomogeneous vector space. 
Then, any $G$-invariant rational function $F\in\C(V)^G$ is constant. 
\end{proposition}
\begin{proof}
By definition, there exists a proper algebraic subset $S\subset V$ whose complement $V\setminus S$ is a single open dense $G$-orbit. 
Then, by assumption, $F$ is a constant function on some open dense subset of $V$. 
This implies that $F$ is constant.  
\end{proof}

\begin{proposition}[{\cite[Proposition 3 in Section 4]{MR0430336}}]\label{prop:unique}\hspace{2pt}\\
Let $(G, \rho, V)$ be a prehomogeneous vector space. 
A relative invariant $F$ is uniquely determined up to a constant multiple by its corresponding character. 
In other words, if $F_1\leftrightarrow \chi$ and $F_2\leftrightarrow \chi$, then $F_1=cF_2$ for some $c\in\Cstar$. In particular, any relative invariant is a homogeneous rational function.  
\end{proposition}
\begin{proof}
If $F_1\leftrightarrow\chi$ and $F_2\leftrightarrow\chi$ for some $\chi$, 
then clearly, $\fr<F_1/F_2>$ is a $G$-invariant rational function. 
Thus, by Proposition \ref{prop:const}, it is a constant. 
Let $F$ be a relative invariant corresponding to $\chi$. 
Then, for each $t\in\Cstar$, we clearly have $F_t(\bm{x}):=F(t\bm{x})\leftrightarrow\chi$. 
Thus, there exists a constant $c_t\in\Cstar$ such that $F(t\bm{x})=c_t\cdot F(\bm{x})$ in $\C(V)$. 
This implies that $F$ is homogeneous. 
\end{proof}

As stated in the following, the Hessian determinant of any relative invariant is also a relative invariant. 
\begin{lemma}\label{lem:relhess}
Let $(G, \rho, V)$ be a prehomogeneous vector space. 
If $F$ is a relative invariant  corresponding to some character $\chi$, then $\det H_F$ is a relative invariant corresponding to the character $\chi^N\cdot (\det)^{-2}$, where $N=\dim V$ and
 $\det : G\to\Cstar : g\mapsto\det(\rho(g))$. 
\end{lemma}
\begin{proof}(cf.\ {\cite[Proof of Proposition 8 in Section 4]{MR0430336}})
By choosing a basis of $V$, we may assume that $V=\C^N$ and $G\subseteq GL_N(\C)$. For $g=(g_{k\ell})\in G$, we have 
\begin{align*}
\fr<\partial^2/\partial x_i\partial x_j>\left(F(g\bm{x})\right)&=\fr<\partial/\partial x_i>\sum_{k=1}^N{\fr<\partial F/\partial x_k>(g\bm{x})\cdot\fr<\partial \left(\sum_{\ell=1}^N{g_{k\ell}x_{\ell}}\right)/\partial x_j>}\\
&=\sum_{k=1}^N{g_{kj}\fr<\partial/\partial x_i>\left(\fr<\partial F/\partial x_k>(g\bm{x})\right)}\\
&=\sum_{k, \ell}{g_{\ell, i}\cdot\fr<\partial^2 F/\partial x_\ell\partial x_k>(g\bm{x})\cdot g_{kj}}.\\
\end{align*}
Then, as a matrix, we have 
\[\left(\fr<\partial^2/\partial x_i\partial x_j>\left(F(g\bm{x})\right)\right)_{i, j}=g^T\left(\fr<\partial^2 F/\partial x_k\partial x_\ell>(g\bm{x})\right)_{k, \ell}g.\]
Since $F(g\bm{x})=\chi(g)F(\bm{x})$, the Hessian matrix $H_F(g\bm{x})$ is 
\[H_F(g\bm{x}):=\left(\fr<\partial^2 F/\partial x_i\partial x_j>(g\bm{x})\right)_{i, j}=\chi(g)\cdot (g^T)^{-1}\left(\fr<\partial^2 F/\partial x_i\partial x_j>(\bm{x})\right)g^{-1}.\]
Then we have $\det H_F(g\bm{x})=\chi(g)^N\cdot (\det g)^{-2}\det H_F(\bm{x})$. 
This means $\det H_F(\bm{x})$ is a relative invariant corresponding to the character $\chi^N\cdot (\det)^{-2}$. 
\end{proof}

Below, let $\left\langle \chi_1, \ldots, \chi_\ell \right\rangle$ be the abelian group generated by characters $\chi_1, \ldots, \chi_\ell$. We say that $\chi_1, \ldots, \chi_\ell$ are {\it multiplicatively independent} if $\left\langle \chi_1, \ldots, \chi_\ell \right\rangle$ is a free abelain group of rank $\ell$.

\begin{lemma}[cf.\ {\cite[Lemma 4 in Section 4]{MR0430336}}]\label{tekitou}\hspace{2pt}\\
Let $(G, \rho, V)$ be a triplet and $F_1, \ldots ,F_\ell$ be relative invariants corresponding to some characters $\chi_1, \ldots ,\chi_\ell\in\Hom(G, \Cstar)$ respectively.  
If $\chi_1, \ldots, \chi_\ell$ are {\it multiplicatively independent}, then \\
$F_1, \ldots ,F_\ell$ is algebraically independent over $\C$. 
\end{lemma}
\begin{proof}
Assume $F_1, \ldots ,F_\ell$ are algebraically dependent. By definition, there exist monomials $\Phi_k(F_1, \ldots, F_\ell):=a_kF_1^{d_{k_1}}\cdots F_{\ell}^{d_{k\ell}} \ (1\leq k\leq s)$ of $F_1, \ldots ,F_\ell$ such that they are linearly dependent over $\C$ and (we can assume) any $s-1$ of them are linearly independent over $\C$ ($s\geq2$). Then, $\Phi_k(F_1, \ldots, F_\ell)$ is clearly a relative invariant corresponding to the character $\mu_k:=\chi_1^{d_{k1}}\cdots\chi_\ell^{d_{k\ell}}$. 
This implies that if $(c_1, \ldots, c_s)\in W:=\{(c_1, \ldots, c_s)\in\C^s \ | \ \sum_{k=1}^s{c_k\Phi_k(F_1, \ldots, F_\ell)}=0\}$, then $(c_1\mu_1(g), \ldots, c_s\mu_s(g))\in W \ (g\in G)$. Since $\dim W=1$, we have $\mu_1=\cdots=\mu_s$. 
On the other hand, any $s-1$ of $\Phi_k(F_1, \ldots, F_\ell) \ (1\leq k\leq s)$ are linearly independent, in particular, for any $1\leq p\neq q\leq s$, we have $(d_{p1}, \ldots, d_{p\ell})\neq (d_{q1}, \ldots, d_{q\ell})$. Then, by assumption, $\chi_1, \ldots, \chi_\ell$ are multiplicatively independent, in particular, if $1\leq p\neq q\leq s$, then $\mu_p\neq \mu_q$. This is a contradiction. 
\end{proof}
\begin{proposition}[cf.\ {\cite[Proposition 5 in Section 4]{MR0430336})}]\hspace{2pt}\\
Let $(G, \rho, V)$ be a prehomogeneous vector space and $S$ be its singular set. 
Let $S_1, \ldots ,S_\ell$ be all codimension one irreducible components of $S$ and $F_i$ be the defining irreducible polynomial of each $S_i$. 
Then, $F_1, \ldots ,F_\ell$ are relative invariants corresponding to some multiplicatively independent characters $\chi_1, \ldots, \chi_\ell$, in particular, $F_1, \ldots ,F_\ell$ are algebraically independent over $\C$. 
Moreover, any relative invariant $F$ can be expressed as $F=cF_1^{m_1}\cdots F_{\ell}^{m_\ell} \ (c\in\C, m_i\in\Z)$. 
In particular, $X_1(G)=\left\langle \chi_1, \ldots, \chi_\ell \right\rangle$ is a free abelian group of rank $\ell$.   
\end{proposition}
\begin{proof}
First, we prove that each $F_i$ is a relative invariant. 
Since $G$ is connected (i.e., irreducible) and $S_i$ is irreducible, the Zariski closure $\overline{\rho(G)\cdot S_i}$ of the image of the multiplication morphism $G\times S_i\to S$ is also irreducible. 
Since $(S_i\subseteq) \ \overline{\rho(G)\cdot S_i} \ (\subset S)$ is irreducible, we have $\overline{\rho(G)\cdot S_i}=S_i$, in particular, ${\rho(G)\cdot S_i}=S_i$. 
This implies that for each $g\in G$, the vanishing loci of two irreducible polynomials $F_i(\bm{x})$ and $F_i(\rho(g)^{-1}\bm{x})$ are the same. 
For each $g\in G$, there exists $\chi_i(g)\in\Cstar$ such that $F_i(\rho(g)\bm{x})=\chi_i(g)F_i(\bm{x})$. 
Then, $\chi_i$ is a character, and $F_i$ is a relative invariant corresponding to $\chi_i$. 
Next, we show that 
$\chi_1, \ldots, \chi_\ell$ are multiplicatively independent. 
If not so, there exists a $(d_1, \ldots, d_\ell)\in\Z^{\ell}\setminus\{0\}$ such that $\chi_1^{d_1}\cdots\chi_\ell^{d_\ell}=1$. 
We may assume $d_1\neq0$. 
Then, $F_1^{-d_1}$ and $F_2^{d_2}\cdots F_\ell^{d_\ell}$ are relative invariants corresponding to the same character $\chi_1^{-d_1}=\chi_2^{d_2}\cdots\chi_\ell^{d_\ell}$. 
By Proposition \ref{prop:unique}, 
$F_1^{-d_1}$ and $F_2^{d_2}\cdots F_\ell^{d_\ell}$ are same up to constant multiple, 
however this contradicts to the irreducibility of $F_i$ and $F_i\neq F_j \ (i\neq j)$. 
The algebraically independence of $F_1, \ldots, F_\ell$ is followed by Lemma \ref{tekitou}. 
Since the vanishing locus of any relative invariant $F$ is $G$-invariant proper subset of $V$, it is a subset of $S$. 
This implies that $F$ is some products of $F_1, \ldots, F_\ell$. 
\end{proof}

To consider when $(G, \rho, V)$ is an irreducible representation, we note the following fundamental theorem by Cartan on irreducible representations. 
\begin{theorem}[Cartan cf.\ {\cite[Theorem 1 in Section 1]{MR0430336}}]\hspace{2pt}\\
Let $(G, \rho, V)$ be a triplet. 
Assume that $d\rho: \gerg\to\gergl(V)$ is an irreducible representation. 
Then, its image $d\rho(\gerg)$ is reductive and isomorphic to one of the following: 
\begin{enumerate}
\item $\gergl_1(\C)\oplus\gerg_1\oplus\cdots\oplus\gerg_s$, where $\gerg_i$ is a simple Lie algebra. 
\item $\gerg_1\oplus\cdots\oplus\gerg_s$, where $\gerg_i$ is  a simple Lie algebra. 
\end{enumerate}
\end{theorem}
By this theorem, we have the following description of $\rho(G)$. 
\begin{corollary}\label{cor:Cartan}
In the above setting, $\rho(G)$ is reductive and isomorphic to one of the following:  
\begin{enumerate}
\item $GL_1(\C)\times G_1\times\cdots\times G_s$, where $G_i$ is an algebraic group whose Lie algebra is simple (in particular, its center $Z(G_i)$ is finite). 
\item $G_1\times\cdots\times G_s$, where $G_i$ is an algebraic group whose Lie algebra is simple (In particular, its center $Z(G_i)$ is finite).  
\end{enumerate}
\end{corollary}

As noted in Remark \ref{rem:subgroup} (2), 
we can think $X_1(G)$ is a subgroup of $\Hom(\rho(G), \Cstar)$. 
Since $\rho(G)$ is reductive, the quotient $\rho(G)/Z(\rho(G))$ by the center $Z(\rho(G))$ is a semi-simple algebraic group. 
Then, the following is exact:
\[0\longrightarrow\Hom(\rho(G)/Z(\rho(G)), \Cstar)\longrightarrow\Hom(\rho(G), \Cstar)\longrightarrow\Hom(Z(\rho(G), \Cstar)). \]  
As the character group of a semi-simple group is trivial, the natural linear map $\Hom(\rho(G), \Cstar)\hookrightarrow\Hom(Z(\rho(G)), \Cstar)$ is injective.  
Thus, we can think as $X_1(G)\subset\Hom(Z(\rho(G)), \Cstar)$. 
Now, since $\rho$ is irreducible, by Corollary \ref{cor:Cartan}, we have
\[\Hom(Z(\rho(G)), \Cstar)\cong \Z\times G_{\text{finite}} \ \text{or} \ G_{\text{finite}},\]
where $G_{\text{finite}}$ is a finite abelian group. 
As $X_1(G)$ is a free abelian group of rank $\ell$, where $\ell$ is the number of irreducible components of  codimension one of the singular set $S$. 
In particular, we have the following. 
\begin{proposition}[cf.\ {\cite[Proposition 12 in Section 4]{MR0430336}}]\label{prop:therelative}\hspace{2pt}\\
Let $(G, \rho, V)$ be an irreducible prehomogeneous vector space. 
Then there is at most one irreducible relative invariant $F$ up to constant multiple. 
In particular, any  relative invariant is in the form of $cF^m$ for $c\in\C$ and $m\in\Z$. 
\end{proposition}

\begin{definition}[cf.\ {\cite[Definition 13 in Section 4]{MR0430336}}]\hspace{2pt}\\
Let $(G, \rho, V)$ be an irreducible prehomogeneous vector space. 
We call $F$ (appeared in Proposition \ref{prop:therelative}) {\it the} relative invariant of $(G, \rho, V)$, which is defined up to constant multiple.
\end{definition}

We say a prehomogeneous vector space $(G, \rho ,V)$ is {\it regular} when there exists a relative invariant $F\in\C(V)$ such that its Hessian determinant $\det H_F$ is not identically zero on $V$ (\cite[Definition 7 in Section 4]{MR0430336}). 
Then by Lemma \ref{lem:relhess}, we have the following key identity of the Hessian of the relative invariant when $(G, \rho, V)$ is regular. 
We learn this corollary from \cite[Remark 3.5]{MR2431661}.  
\begin{corollary}\label{cor:reg}
Let $(G, \rho, V)$ be a regular {irreducible} prehomogeneous vector space of dimension $n$. 
Assume that the degree of the relative invariant $F$ is $r$. 
Then, the Hessian of $F$ is in the form of  
\[\det H_F=cF^{\fr<n(r-2)/r>},\]
where $c\in\Cstar$ is a constant.  
\end{corollary}

\subsection{Strict log-concavity of homogeneous polynomials}
Let $F$ be a homogeneous polynomial of degree $r$ in $n$ variables with real coefficients, where $r\geq3$. 
Here we consider log-concavity of $F$. 
For a symmetric matrix $A$, $A\succeq0$ denotes that $A$ is positive semi-definite, and $A\succ0$ denotes that $A$ is positive definite. 
Now we define strict log-concavity. 

\begin{definition}[(strict) log-concavity]\hspace{2pt}\\
We say that $F$ is {\it log-concave} (resp.\ {\it strictly log-concave}) at $\bm{a}\in\R^n$ if 
\[(-FH_F+(\nabla F)^T(\nabla F))|_{\bm{x}=\bm{a}}\succeq 0 \  (\text{resp.\ } \ \succ0). \]
\end{definition}

For technical reasons, we introduce strict ``homogeneous" log-concavity which is stronger than strict log-concavity.  
We will not, however, use this notion essentially until the final section; therefore it is not a problem to replace (strict) homogeneous log-concavity with (strict) log-concavity until then.

\begin{definition}[(strict) homogeneous log-concavity]\hspace{2pt}\\
We say that $F$ is {\it homogenenous log-concave} (resp.\ {\it strictly homogenenous log-concave}) at $\bm{a}\in\R^n$ if for any $s\geq\fr<r-1/r>$ (resp. $s>\fr<r-1/r>$), 
\[(-FH_F+s(\nabla F)^T(\nabla F))|_{\bm{x}=\bm{a}}\succeq 0 \  (\text{resp.\ } \ \succ0).\]
\end{definition}

As remarked in \cite[Example 1.11.2]{MR2683227}, 
$F$ is (strictly) homogeneous log-concave at $\bm{a}\in\R^n$ if and only if $F^{\fr<1/k>}$ is log-concave at $\bm{a}$ for any $k>r$. 

Clearly, strict (homogeneous) log-concavity implies (homogeneous) log-concavity.  

{From here, we assume that $F$ is a homogeneous polynomial with positive coefficients.} 
One of the important properties of strictly log-concave homogeneous polynomial $F$ with positive coefficients is that its Hessian $H_F$ is non-degenerate, moreover it has only one positive eigenvalue. To prove this, we note Cauchy's interlacing theorem.  

\begin{theorem}[Cauchy's interlacing Theorem {\cite[Corollary 4.3.9]{MR2978290}}]\label{thm:Cauchy}\hspace{2pt}\\
For a real symmetric $n\times n$ matrix $A$ with eigenvalues $\alpha_1\geq\cdots\geq\alpha_n$ and a vector $\bm{v}\in\R^n$, the eigenvalues $\alpha_1\geq\cdots\geq\alpha_n$ interlace the eigenvalues $\beta_1\geq\cdots\geq\beta_n$ of $B:=A+\bm{v}\bm{v}^T$. 
That is, 
\[\beta_1\geq\alpha_1\geq\beta_2\geq\cdots\geq\alpha_{n-1}\geq\beta_n\geq\alpha_n.\] 
\end{theorem}

\begin{corollary}\label{cor:eigen}
{Let $F$ be a homogeneous polynomial with positive coefficients. } If $F$ is strictly log-concave at $\bm{a}\in(\R_{>0})^n$, 
then $H_F|_{\bm{x}=\bm{a}}$ has exactly $n-1$ negative eigenvalues and exactly one positive eigenvalue. 
In particular, 
\begin{align*}
(-1)^{n-1}(\det H_F)|_{\bm{x}=\bm{a}}>0. 
\end{align*}
\end{corollary}
\begin{proof}
We set $A=\left.\left(-FH_F\right)\right|_{\bm{x}=\bm{a}}$ and $B=(-FH_{F}+(\nabla F)^T(\nabla F))|_{\bm{x}=\bm{a}}$, and denote their eigenvalues as $\alpha_1\geq\cdots\geq\alpha_n$ and $\beta_1\geq\cdots\geq\beta_n$ respectively. 
Since $F$ is strictly log-concave at $\bm{a}\in(\R_{>0})^n$, we have $\beta_n>0$. 
Hence it follows from Cauchy's interlacing theorem that eigenvalues $\alpha_1, \ldots, \alpha_{n-1}$ are positive. 
On the other hand, {we have $\tr A=\sum_{i=1}^n{\alpha_i}=-\left. \left(F\sum_{i=1}^n{\fr<\partial^2 F/\partial x_i^2>}\right) \right|_{\bm{x}=\bm{a}}\leq0$.} Thus, $\alpha_n$ should be negative. 
Hence $\left.\left(-FH_F\right)\right|_{\bm{x}=\bm{a}}$ has exactly $n-1$ positive eigenvalues and exactly one negative eigenvalue. 
Since $F$ is a polynomial with positive coefficients, we have $F(\bm{a})>0$ for any point $\bm{a}\in(\R_{>0})^n$. 
Therefore $H_{F}|_{\bm{x}=\bm{a}}$ has exactly $n-1$ negative eigenvalues and exactly one positive eigenvalue. 
\end{proof}

For $F$, we define
\begin{align*}
F_0&=F|_{x_k=0}\in\R[x_1, \ldots, \hat{x_k}, \ldots ,x_N], \\
F_k&=\fr<\partial F/\partial x_k>\in\R[x_1, \ldots, \hat{x_k}, \ldots, x_N]. 
\end{align*}
Note that $F=F_0+x_kF_k$ in this case. 

The following lemma looks rather technical, however this gives a relationship between (strict) homogeneous log-concavity of $F$ and (strict) homogeneous log-concavity of $F_{0}$ and $F_{k}$. 

\begin{lemma}\label{cor:equivalent}
If $F_0(a_1, \ldots, \hat{a_k}, \ldots ,a_N)\neq0$ and $F_k(a_1, \ldots, \hat{a_k}, \ldots ,a_N)\neq0$ for $\bm{a}\in\R_{\geq0}^N$, then the following are equivalent for any $s\geq\fr<r-1/r>$ (resp.\ $s>\fr<r-1/r>$). 
\begin{enumerate}
\item[(i)] $(-FH_F+s(\nabla F)^T(\nabla F))|_{\bm{x}=\bm{a}}\succeq0$ (resp.\ $\succ0$).
\vspace{4pt}
\item[(ii)] $\left.\left(\begin{array}{l}sx_kF_0F_k\left(-F_kH_{F_k}+\fr<2s-1/s>(\nabla F_k)^T\nabla F_k\right)\\
+sF_k^2(-F_0H_{F_0}+s(\nabla F_0)^T\nabla F_0)\\
-(sF_k\nabla F_0-F_0\nabla F_k)^T(sF_k\nabla F_0-F_0\nabla F_k)\end{array}\right)\right|_{\bm{x}=\bm{a}}\succeq0 \ (\text{resp.\ } \succ 0)$.
\end{enumerate}
\end{lemma}
\begin{proof}
For conciseness, we will omit {$|_{\bm{x}=\bm{a}}$}. 
We show the equivalence for only positive definiteness (the argument is similar for positive semi-definiteness). We may assume $k=1$. Since we have $F=F_0+x_1F_1$, we can compute $-FH_F+s(\nabla F)^T(\nabla F)$ as follows. 
Here, note that the $(1, 1)$-component of the Hessian matrix of $F$ is 0 since $F$ is multi-affine. 
\begin{multline*}
-FH_F+s(\nabla F)^T(\nabla F)
=-F
\left(
\begin{array}{c|ccc}
0&&\nabla F_1&\\\cline{1-4}
&&&\\
(\nabla F_1)^T&&H_{F_0}+x_1H_{F_1}&\\
&&&
\end{array}
\right) \\
+s
\left(
\begin{array}{c|ccc}
F_1^2&&F_1(\nabla F_0+x_1\nabla F_1)^T&\\\cline{1-4}
&&&\\
F_1(\nabla F_0+x_1\nabla F_1)^T&&(\nabla F_0+x_1\nabla F_1)^T(\nabla F_0+x_1\nabla F_1)&\\
&&&
\end{array}
\right). 
\end{multline*}
Then, for any $\tilde{\bm{y}}=\left(\begin{array}{c|ccc}y_1&&\bm{y}&\end{array}\right)^T\in\R^N\setminus\{\bm{0}\}$, we have
\begin{equation*}
\begin{split}
\tilde{\bm{y}}^T(-F&H_F+s(\nabla F)^T(\nabla F))\tilde{\bm{y}}\\
&=-F\{2y_1(\nabla F_1\bm{y})+\bm{y}^TH_{F_0}\bm{y}+x_1(\bm{y}^TH_{F_1}\bm{y})\}\\
&\quad\quad+s\{F_1^2y_1^2+2y_1F_1(\nabla F_0\bm{y}+x_1(\nabla F_1\bm{y}))+(\nabla F_0\bm{y}+x_1(\nabla F_1\bm{y}))^2\}\\
&=(sF_1^2)y_1^2+2\left\{-F(\nabla F_1\bm{y})+sF_1(\nabla F_0\bm{y}+x_1(\nabla F_1\bm{y}))\right\}y_1\\
&\quad\quad+\bm{y}^T\left(-F(H_{F_0}+x_1H_{F_1})+s(\nabla F_0+x_1\nabla F_1)^T(\nabla F_0+x_1\nabla F_1)\right)\bm{y}.
\end{split}
\end{equation*}
If $\bm{y}=0$, then $y_1\neq0$, so $\tilde{\bm{y}}^T(-FH_F+s(\nabla F)^T(\nabla F))\tilde{\bm{y}}=(sF_1^2)y_1^2>0$. Thus, $(-FH_F+s(\nabla F)^T(\nabla F))$ is positive definite if and only if for any $y_1\in\R$ and $\bm{y}\neq\bm{0}$, $\alpha y_1^2+2\beta y_1+\gamma>0$, where 
\begin{align*}
\alpha&=sF_1^2>0, \\
\beta&=-F(\nabla F_1\bm{y})+sF_1(\nabla F_0\bm{y}+x_1(\nabla F_1\bm{y})),\\
\gamma&=\bm{y}^T\left(-F(H_{F_0}+x_1H_{F_1})+s(\nabla F_0+x_1\nabla F_1)^T(\nabla F_0+x_1\nabla F_1)\right)\bm{y}.
\end{align*}
Since this is equivalent to $\alpha\gamma-\beta^2>0$ for any $\bm{y}\in\R^{n-1}\setminus{\bm{0}}$, then we have  
\begin{align*}
\begin{split}
0 < &
(sF_1^2)\{-F(\bm{y}^TH_{F_0}\bm{y}+x_1\bm{y}^TH_{F_1}\bm{y})+s(\nabla F_0\bm{y}+x_1\nabla F_1\bm{y})^2\}\\
&\hspace{10em}-\{-F(\nabla F_1\bm{y})+sF_1(\nabla F_0\bm{y}+x_1(\nabla F_1\bm{y}))\}^2\\
&=-(sF_1^2)F(\bm{y}^TH_{F_0}\bm{y}+x_1\bm{y}^TH_{F_1}\bm{y})-F^2(\nabla F_1\bm{y})^2\\
&\hspace{10em}+2sF_1F(\nabla F_1\bm{y})(\nabla F_0\bm{y}+x_1(\nabla F_1\bm{y})).
\end{split}
\end{align*}
Dividing both sides by $F$, we have 
\begin{align*}
\begin{split}
0 < & 
-(sF_1^2)(\bm{y}^TH_{F_0}\bm{y}+x_1\bm{y}^TH_{F_1}\bm{y})-(F_0+x_1F_1)(\nabla F_1\bm{y})^2\\
&\hspace{10em}+2sF_1(\nabla F_1\bm{y})(\nabla F_0\bm{y}+x_1(\nabla F_1\bm{y}))\\
&=sx_1F_1\left\{-F_1(\bm{y}^TH_{F_1}\bm{y})+\fr<2s-1/s>(\nabla F_1\bm{y})^2\right\}\\
&\hspace{8em}+\left\{-sF_1^2(\bm{y}^TH_{F_0}\bm{y})-F_0(\nabla F_1\bm{y})^2+2sF_1(\nabla F_1\bm{y})(\nabla F_0\bm{y})\right\}\\
&=sx_1F_1\left\{-F_1(\bm{y}^TH_{F_1}\bm{y})+\fr<2s-1/s>(\nabla F_1\bm{y})^2\right\}\\
&\hspace{10em}+\fr<sF_1^2/F_0>\left\{-F_0(\bm{y}^TH_{F_0}\bm{y})+s(\nabla F_0\bm{y})(\nabla F_0\bm{y})\right\}\\
&\hspace{10em}-\fr<s^2F_1^2/F_0>(\nabla F_0\bm{y})^2-F_0(\nabla F_1\bm{y})^2+2sF_1(\nabla F_1\bm{y})(\nabla F_0\bm{y})\\
&=sx_1F_1\left\{-F_1(\bm{y}^TH_{F_1}\bm{y})+\fr<2s-1/s>(\nabla F_1\bm{y})^2\right\}\\
&\hspace{10em}+\fr<sF_1^2/F_0>\left\{-F_0(\bm{y}^TH_{F_0}\bm{y})+s(\nabla F_0\bm{y})(\nabla F_0\bm{y})\right\}\\
&\hspace{10em}-\fr<1/F_0>\left\{sF_1(\nabla F_0\bm{y})-F_0(\nabla F_1\bm{y})\right\}^2\\
&=\fr<1/F_0>\bm{y}^T\left(\begin{array}{l}sx_1F_0F_1\left(-F_1H_{F_1}+\fr<2s-1/s>(\nabla F_1)^T\nabla F_1\right)\\
+sF_1^2(-F_0H_{F_0}+s(\nabla F_0)^T\nabla F_0)\\
-(sF_1\nabla F_0-F_0\nabla F_1)^T(sF_1\nabla F_0-F_0\nabla F_1)\end{array}\right)\bm{y}. 
\end{split}
\end{align*}
After multiplying both sides by $F_0$, we complete the proof of the equivalence of (i) and (ii).
\end{proof}

By Lemma \ref{cor:equivalent}, we prove the following which is important in the proof of our main theorem (Theorem \ref{cor:matroidver}). 

\begin{corollary}\label{cor:van}
Let $F\in\R[x_1, \ldots, x_N]$ be a multi-affine homogeneous polynomial of $\deg F=r\geq3$ with positive coefficients. 
For a subset $I$ of $[N]$ and $0\leq k \leq N$, we define
\begin{align*}
C^{N-k}_{I>0}=\{(z_{k+1}, \ldots, z_N)\in\R^{N-k}_{\geq0} \ |  \ z_j\geq0 \ (j\notin I), \ z_i>0 \ (i\in I)\}. 
\end{align*}
We assume that $F$ is strictly homogeneous log-concave on $C^N_{I>0}$. 
If  
\begin{equation}
\fr<\partial F/\partial x_{1}>\neq0, \fr<\partial F|_{x_{1}=0}/\partial x_{2}>\neq0, \ldots ,\fr<\partial F|_{x_{1}=\cdots=x_{{k-1}}=0}/\partial x_{k}>\neq0
\end{equation}
holds as a polynomial for some $0\leq k\leq N-r$, then $F|_{x_{1}=\cdots=x_{k}=0}\in\R[x_{k+1}, \ldots, x_N]$ is strictly homogeneous log-concave on $C^{N-k}_{I>0}$. 
\end{corollary}
\begin{proof}
We show this by induction on $k$. 
In the case where $k=0$, the claim is obvious by the assumption. 
For $1\leq k\leq N-r$, by the induction hypothesis, 
$F|_{x_{1}=\cdots=x_{{k-1}}=0}$ is strictly homogeneous log-concave on $C^{N-k+1}_{I>0}$. 
Let 
\begin{align*}
f=F|_{x_{1}=\cdots=x_{{k-1}}=0}\in\R[x_k, \ldots, x_N]. 
\end{align*}
Applying Lemma \ref{cor:equivalent} to $f$ and $\bm{a}=\left(\begin{array}{c|ccc}0&&\bm{\overline{z}}&\end{array}\right)^T\in C_{I>0}^{N-k+1}$ for any $\bm{\overline{z}}\in C_{I>0}^{N-k}$, 
we have 

\begin{align*}
\left.
\left(
\begin{array}{l}
sf_k^2(-f_0H_{f_0}+s(\nabla f_0)^T(\nabla f_0))\\
-(sf_k\nabla f_0-f_0\nabla f_k)^T(sf_k\nabla f_0-f_0\nabla f_k)
\end{array}
\right)
\right|
_{(x_{k+1}, \ldots, x_n)=\bm{\overline{z}}}
\succ 0,
\end{align*}
where $f_0:=F|_{x_{1}=\cdots=x_{{k}}=0}$, $f_k:=\fr<\partial F|_{x_{1}=\cdots=x_{{k-1}}=0}/\partial x_{k}>$.  
Note that by assumption, they are not identically zero as polynomials. 
In particular, for any $\bm{\overline{z}}\in C_{I>0}^{N-k}$, we have 
\[\left.\left(-f_0H_{f_0}+s(\nabla f_0)^T(\nabla f_0)\right)\right|_{(x_{k+1}, \ldots, x_N)=\bm{\overline{z}}}\succ0.\]
This completes the proof.  
\end{proof}

\section{Matroids}\label{sec:Matroid}

In this section, we provide basic terms of a matroid. 
{The best general reference for matroid theory is \cite{MR2849819}.} 

\begin{definition}[Matroid]\label{def:basis}\hspace{2pt}\\
A {\it matroid} $M$ is an ordered pair $(E, \calB)$ consisting of a finite set $E$ and a collection $\calB$ of subsets of $E$ satisfying the following properties: 
\begin{itemize}
\item $\calB\neq\emptyset$.
\item If $B_1$ and $B_2$ are in $\calB$ and $x\in B_1\setminus B_2$, then there is an element $y\in B_2\setminus B_1$ such that $\{y\}\cup(B_1\setminus\{x\})\in\calB$. 
\end{itemize}
In this case, we call each $B\in\calB$ a {\it basis} of $M$. 
\end{definition}

\begin{example}[Graphic matroid]\label{Graphic matroid}\hspace{2pt}\\
For any finite graph $\Gamma=(V, E)$ with the vertex set $V$ and the edge set $E$, we call a subgraph $T\subseteq \Gamma$ a {\it spanning tree} in $\Gamma$ if $T$ does not contain any cycles and $T$ passes through all vertices of $\Gamma$. 
Let $\calB_\Gamma$ be the set of all spanning trees in $\Gamma$. Then $M(\Gamma)=(E, \calB_\Gamma)$ is a matroid. 
These matroids are called {\it graphic matroids}.  
\end{example}

\begin{remark}
If $M$ is a graphic matroid, then there exists a connected graph $\Gamma$ such that $M(\Gamma)$ is isomorphic to $M$.  
\end{remark}

\begin{example}[Submatroid]\hspace{2pt}\\
Let $M=(E, \calB)$ be a matroid. 
For $E'\subset E$, we define $\calB'$ by $\calB'=\Set{B\in\calB|B\subset E'}$. 
Then $M'=(E', \calB')$ is a matroid. 
We call $M'$ a submatroid of $M$. 
\end{example}

Let $M=(E, \calB)$ be a matroid. 
We call each subset of a basis of $M$ an \textit{independent set} of $M$ 
and call each subset of $E$, which is not contained in any basis, a \textit{dependent set} of $M$. 
A minimal dependent set of $M$ is called a \textit{circuit} of $M$. 
We say that $C$ is an $n$-circuit if $C$ is a circuit and $C$ has $n$ elements.  
In particular, we call each $1$-circuit a loop. 
We call an element $e$ a {\it coloop} of $M$ if $\{e\}$ is contained in each basis of $M$. 
We say that a matroid $M$ is \textit{simple} if there is neither a $1$-circuit nor $2$-circuit. 

We can directly prove the following from the definition of the basis. 

\begin{proposition}\label{the number of basis are the same}
Let $M$ be a matroid with the basis set $\calB$. 
If $B$ and $B'$ are basis of $M$, then the number of elements of them are the same. 
In other words, if $B, B'\in \calB$, then $|B|=|B'|$. 
\end{proposition}

We say that a matroid $M$ has \textit{rank} $r$ if the number of elements of a basis of $M$ is $r$.  
The rank of $M$ is denoted by $\rank M$. 

\begin{definition}[Basis generating function]\label{def:basispoly}\hspace{2pt}\\
For any matroid $M=(E, \calB)$, we define the {\it basis generating function} $F_M(\bm{x})$ of $M$ by 
\[F_M(\bm{x})=\sum_{B\in\calB}{\prod_{i\in B}{x_i}}.\]
\end{definition}

By Definition and Proposition \ref{the number of basis are the same}, 
for a matroid $M=(E, \calB)$ of rank $r$, 
its basis generating function $F_M(\bm{x})$ is a multi-affine homogeneous polynomial of degree $r$ in $|E|$ variables with positive coefficients. 
Moreover, for any $e\in E$ which is not a loop or a coloop, we have 
\[F_M(\bm{x})=F_{M\setminus e}(\bm{x})+x_eF_{M/e}(\bm{x}),\]
where $M\setminus e$ (resp.\ $M/e$) is the deletion (resp.\ contraction) of $M$ with respect to $e$ (see \cite{MR2849819} for the definitions). 
In particular, if matroid $M_0$ is obtained by deleting some elements $e_1, \ldots, e_k\in E$ from $M$, 
then we have 
\begin{align*}
F_{M_0}=F_M|_{x_{e_1}=\cdots=x_{e_k}=0}. 
\end{align*}

Note that for any matroid $M$ on $[n]=\Set{1,2, \ldots, n}$, 
(not necessarily strict) homogeneous log-concavity of $F_M(\bm{x})$ on $\R_{\geq0}^n$ is already known in \cite[Theorem 4.2]{AGV2018} as stated below. 
Precisely speaking, they show log-concavity in their paper, however by carefully reading their proof, one can easily show homogeneous log-concavity of $F_M(\bm{x})$ on $\R_{\geq0}^n$.

\begin{theorem}[{\cite[Theorem 4.2]{AGV2018}}]\label{thm:AGVs}\hspace{2pt}\\
For any matroid $M$, $F_M(\bm{x})$ is homogeneous log-concave on $\R^n_{\geq0}$. 
In other words,  
\[{\left.\left(-F_MH_{F_M}+s(\nabla F_M)^T(\nabla F_M)\right)\right|_{\bm{x}=\bm{a}}}\succeq0\]
for any $\bm{a}\in\R_{\geq0}^n$ and $s\geq\fr<r-1/r>$.  
\end{theorem}

\begin{remark}
In \cite[Theorem 4.2]{AGV2018}, the authors show that 
$F_M(\bm{x})$ satisfies {\it complete log-concavity}, 
i.e., for any $\bm{v_1}, \ldots, \bm{v_k}\in\R_{\geq0}^n \ (0\leq k\leq r-2)$, 
$\partial_{\bm{v_1}}\cdots\partial_{\bm{v_k}}F_M(\bm{x})$ is log-concave on $\R_{\geq0}^n$. 
\end{remark}

\begin{remark}
If $M$ is not simple, then $\det (-F_MH_{F_M}+s(\nabla F_M)^T(\nabla F_M))$ is identically zero, in particular, it cannot be positive definite at any point in $\R^n$. In fact, we assume $M$ has a loop $e$ or parallel elements $\{e_1, e_2\}$. 
In the former case, by definition, $\fr<\partial/\partial x_e>F_M=0$, in particular, $\det H_{F_M}=0$. 
In the latter case, we can express $F_M$ like 
\begin{align*}
F_M=F_M|_{x_{e_1}=x_{e_2}=0}+(x_{e_1}+x_{e_2})G, 
\end{align*} 
where 
\begin{align*}
G=G(x_1, \ldots, \hat x_{e_1}, \ldots, \hat x_{e_2}, \ldots, x_n)=\fr<\partial F_{M}/\partial x_{e_1}>=\fr<\partial F_{M}/\partial x_{e_2}>. 
\end{align*}
Thus we have $\det H_{F_M}=0$. In both cases, we have $\det H_{F_M}=0$. 
As seen in Proposition \ref{identity1}, this implies that $\det (-F_MH_{F_M}+s(\nabla F_M)^T(\nabla F_M))=0$. 
\end{remark}

In the rest of this section, we prepared some lemmas for our main theorem. 

\begin{lemma}\label{lem:star}
Let $M$ be a matroid on $[N]$ of $\rank M=r\geq2$ with no loops (we don't assume $M$ is simple). We consider its basis generating function $F_M(\bm{x})$.  
For any basis $B=\{i_{N-r+1}, \ldots, i_N\}\in\calB$ of $M$ and its complement $\{j_1, \ldots ,j_{N-r}\}$, $F_M(\bm{x})$ satisfies the following $(1\leq k\leq N-r)$. 
\begin{equation}\label{partial differentiation of the basis generating function for a matroid which is obtained from deletion}
\fr<\partial F_{M}|_{x_{j_1}=\cdots=x_{j_{k-1}}=0}/\partial x_{i_k}>\neq0.
\end{equation}
\end{lemma}
\begin{proof}
By the definition of $F_M$, 
we only have to show that 
for each $k$, there exists a basis $B_0\in\calB$ such that $B_0\cap\{i_1, \ldots ,i_{k-1}\}=\emptyset$ and $i_k\in B_0$. 
Below, we will show that there exists some $\ell$ such that we can take $\{i_k\}\cup\{j_{N-r+1}, \ldots ,\hat{j_{\ell}}, \ldots ,j_{N}\}$ as $B_0$. 
In fact, by \cite[Corollary 1.2.6]{MR2849819} there is a unique circuit $C(i_k, B)$ which is contained in $B\cup\{i_k\}$ (so-called the {\it fundamental circuit}). 
Since by definition and assumption, $i_k\in C(i_k, B)$, $i_k$ is not a loop, and $C(i_k, B)$ contains some $j_\ell$. 
Then by \cite[Exercise 1.2.5]{MR2849819}, $B_0:=\{i_k\}\cup\{j_{N-r+1}, \ldots ,\hat{j_{\ell}}, \ldots ,j_{N}\}$ is a basis. 
\end{proof}

Since the basis generating function $F_M(\bm{x})$ of any simple matroid $M$ satisfies the condition \eqref{partial differentiation of the basis generating function for a matroid which is obtained from deletion} by Corollary \ref{cor:van} and Lemma \ref{lem:star}, we have the following. 

\begin{theorem}\label{cor:matroidver}
Let $M$ be a simple matroid on $[N]$ of $\rank M=r\geq3$. 
For any basis $B$, we assume that $F_M$ is strictly homogeneous log-concave on $C_{B>0}^N \ (\subseteq(\R_{>0})^N)$. 
Then for any submatroid $M_0:=M\setminus\{j_1, \ldots, j_k\}$ of rank $r$, $F_{M_0}$ is strictly homogeneous log-concave on $C_{B_0>0}^{N-k} \ (\subseteq(\R_{>0})^{N-k})$ for any basis $B_0$ of $M_0$. 
\end{theorem}
\begin{proof}
Let $B_0:=\{i_{N-r+1}, \ldots, i_{N}\}$ be a basis of $M_0$ (and $M$). 
Since $F_M$ satisfies the condition (*) for $x_{j_1}, \ldots, x_{j_k}$ by Lemma \ref{lem:star}, the polynomial  $F_{M_0}=F_M|_{x_{j_1}=\cdots=x_{j_k}=0}$ is strictly homogeneous log-concave on $C_{B_0>0}^{N-k} \ (\subseteq(\R_{>0})^{N-k})$. 
\end{proof}

\section{Main result}\label{sec:Main result}
{In this section, we will prove our main result. }Our main result is that the Kirchhoff polynomial of each simple graph is strictly log-concave on $\R^{n}_{>0}$ (Theorem \ref{mainresult}). 

First, we define the Kirchhoff polynomial of a graph. 
\begin{definition}[Kirchhoff polynomial]\label{Kirchhoff polynomial}\hspace{2pt}\\
For a connected graph $\Gamma=(V, E)$ with $|E|=n$, we define the {\it Kirchhoff polynomial} of $\Gamma$ by 
\[F_\Gamma(x_1, \ldots, x_n)=\sum_{T\in\calB_\Gamma}{\prod_{i\in T}{x_i}},\]
where $\calB_\Gamma$ is the set of spanning trees in $\Gamma$. 
\end{definition}
The Kirchhoff polynomial can be seen as a special case of the basis generating function of a matroid by Example \ref{Graphic matroid}.

\begin{theorem}[Main result]\label{mainresult}\hspace{2pt}\\
For any simple graph $\Gamma=(V, E)$ with $|V|=r+1\geq3$ and $|E|=n\geq3$, the Kirchhoff polynomial $F_\Gamma(\bm{x})$ is strictly homogeneous log-concave on $(\R_{>0})^n$. 
In other words, 
\[(-F_\Gamma H_{F_\Gamma}+s(\nabla F_{\Gamma})^T\nabla F_\Gamma)|_{\bm{x}=\bm{a}}\succ0\]
for any $\bm{a}\in(\R_{>0})^n$ and $s>\fr<r-1/r>$. 
In particular, $H_{F_\Gamma}|_{\bm{x}=\bm{a}}$ is non-degenerate, with $n-1$ negative eigenvalues and exactly one positive eigenvalue. 
Thus, 
\begin{align*}
(-1)^{n-1}(\det H_{F_\Gamma})|_{\bm{x}=\bm{a}}>0. 
\end{align*}
Moreover, for each spanning tree $T$ in $\Gamma$, $F_\Gamma$ is strictly homogeneous log-concave on $C^n_{T>0}$, 
where
\begin{align*}
C^n_{T>0}=\{\bm{a}\in\R^n_{\geq0} \ | \ z_i>0 \ (i\in T), \ z_j\geq0 \ (j\notin T)\} \ (\supseteq(\R_{>0})^n). 
\end{align*}
\end{theorem}

Now, we will prove our main theorem: 
The Kirchhoff polynomial can be seen as the special case of the basis generating function of a matroid. 
Hence we have log-concavity of the Kirchhoff polynomial by Theorem \ref{thm:AGVs}
and we must only verify the strictness. 
Since we have Proposition \ref{lem:hom1}, 
the Kirchhoff polynomial is strictly log-concave if and only if its Hessian does not vanish. 
Every simple graph is obtained from the complete graph with the same number vertices by cutting edges. 
In other words, every simple graphic matroid is a submatroid of the graphic matroid of the complete graph. 
We can easily find the following corollary by Theorem \ref{cor:matroidver}.  
\begin{corollary}\label{maintheorem1}
Let $\Gamma=(V, E)$ be a simple graph with $|V|=r+1\geq3$ and $|E|=n\geq3$. 
For each spanning tree $T$ in $\Gamma$, we assume that $F_\Gamma$ is strictly homogeneous log-concave on $C^n_{T>0}$. 
Then for any connected subgraph $\Gamma'=(V', E')$ with $|V'|=r+1$ and $|E'|=n-k$, $F_{\Gamma'}$ is strictly homogeneous log-concave on $C_{T'>0}^{n-k} \ (\supseteq(\R_{>0})^{n-k})$ for any basis $T'$ in $\Gamma'$. 
\end{corollary}
Since we have Corollary \ref{maintheorem1}, we only have to show the Hessian does not vanish in the case of the complete graph. 
As stated in Section \ref{sec:Homogeneous polynomials},  
for the relative invariant of an irreducible prehomogeneous vector space, 
its Hessian is in the form of $cF^{m}$. 
We can show that the Kirchhoff polynomial of the complete graph can be realized as the relative invariant. 
Then we have the following. 
\begin{theorem}\label{maintheorem2}
Let $N=\binom{r+1}{2}$. 
We have 
\[\det H_{F_{K_{r+1}}}=(-1)^{N-1}c_r(F_{K_{r+1}})^{N-r-1}, \]
where $c_{r}=2^{N-r}(r-1)$. 
\end{theorem}
Theorem \ref{maintheorem2} implies that for any spanning tree $T$, the Kirchhoff polynomial is strictly log-concave on $C_{T>0}^n$. 
Hence we obtain our main result from Corollary \ref{maintheorem1}.  

In the rest of this section, we study more precisely the Kirchhoff polynomials and give a proof of Theorem \ref{maintheorem2}. 

In general, if a connected graph $\Gamma$ has $r+1$ vertices, then a spanning tree in $\Gamma$ has $r$ edges. 
Hence, the Kirchhoff polynomial of $\Gamma$ with $r+1$ vertices is a homogeneous polynomial of degree $r$. Moreover, the Kirchhoff polynomial is a multi-affine polynomial of which each coefficient is one.

\begin{example}\label{eg:1}
Consider the two hollowing graphs. 
\begin{figure}[htbp]
 \begin{minipage}{0.3\hsize}
  \begin{center}
\begin{tikzcd}[row sep=tiny, column sep=tiny]
\overset{1}{\bullet}\ar[rr, no head]\ar[ddrr, no head]\ar[dd, no head]&&\overset{2}{\bullet}\ar[ddll, no head]\ar[dd, no head]\\
&&\\
\underset{3}{\bullet}\ar[rr, no head]&&\underset{4}{\bullet}\\
&\hspace{10pt}\hspace{10pt}
\end{tikzcd}   
  \end{center}
\vspace{-3ex}
  \caption{$K_{4}$}
  \label{fig:one}
 \end{minipage}
 \begin{minipage}{0.5\hsize}
  \begin{center}
\begin{tikzcd}[row sep=1ex, column sep=1ex]
\overset{1}{\bullet}\ar[rr, no head]\ar[ddrr, no head]\ar[dd, no head]&&\overset{2}{\bullet}\ar[dd, no head]\\&&\\\underset{3}{\bullet}\ar[rr, no head]&&\underset{4}{\bullet}\\&\hspace{10pt}\hspace{10pt}&
\end{tikzcd}
  \end{center}
\vspace{-3ex}
  \caption{$K_{4}\setminus \Set{2,3}$}
  \label{fig:two}
 \end{minipage}
\end{figure}

The number of spanning trees in $K_{4}$ and $K_{4}\setminus\Set{2,3}$ are sixteen and eight, respectively.  
Then the Kirchhoff polynomial of $K_{4}$ is as follows:   
\begin{align*}
F_{K_4}(\bm{x})=
&x_{12}x_{13}x_{14} + x_{12}x_{14}x_{23} + x_{13}x_{14}x_{23} + x_{12}x_{13}x_{24} \\
+ &x_{13}x_{14}x_{24}+ x_{12}x_{23}x_{24} + x_{13}x_{23}x_{24} + x_{14}x_{23}x_{24} \\
+ &x_{12}x_{13}x_{34} + x_{12}x_{14}x_{34}+ x_{12}x_{23}x_{34} + x_{13}x_{23}x_{34} \\
+ &x_{14}x_{23}x_{34} + x_{12}x_{24}x_{34} + x_{13}x_{24}x_{34} + x_{14}x_{24}x_{34}. 
\end{align*}
And the Kirchhoff polynomial of $K_{4}\setminus\Set{2,3}$ is as follows:   
\begin{align*}
F_{K_4\setminus\Set{2,3}}(\bm{x})=
&x_{12}x_{13}x_{14} +  x_{12}x_{13}x_{24} + x_{13}x_{14}x_{24} + x_{12}x_{13}x_{34}\\
+ &x_{12}x_{14}x_{34} + x_{12}x_{24}x_{34} + x_{13}x_{24}x_{34} + x_{14}x_{24}x_{34}. 
\end{align*}
\end{example}

In Example \ref{eg:1}, we can see that the Kirchhoff polynomial of $K_{4}\setminus\Set{2,3}$ is equal to the Kirchhoff polynomial of $K_{4}$ substituting zero to the variable $x_{23}$. 
In general, every Kirchhoff polynomial is obtained from the Kirchhoff polynomial of the complete graph with same number vertices by substituting zero {for} some variables. 

Next, we see that the Kirchhoff polynomial is realized as the determinant of some matrix. 
This is called the Matrix-tree theorem. 
Let $(E_{r})_{ij}$ be an $r\times r$ matrix, where its $(i, j)$-component is one and the others are zero. 
If the size of the matrix is obvious from the context, then we drop the size to $E_{ij}=(E_{r})_{ij}$. 
For a matrix $X$, $X^{(ij)}$ denotes the submatrix of $X$ obtained by removing the $i$th row and the $j$th column.

\begin{definition}[Laplacian]\hspace{2pt}\\
For a graph $\Gamma=(V, E)$ with $|V|=r$, we associate a variable $x_e$ to each edge $e\in E$, and define the {\it Laplacian $L_{\Gamma}$ of $\Gamma$} indexed by vertices by 
\[L_\Gamma=\sum_{e=\{i, j\}\in E}{x_e(E_{ii}-E_{ij}-E_{ji}+E_{jj})}.\]
\end{definition}

The following theorem is a well-known fact. 
For example, see \cite[Theorem VI.29]{MR1813436}. 

\begin{theorem}[The Matrix-Tree Theorem]\label{thm:matrixtree}\hspace{2pt}\\
For a graph $\Gamma$, its Kirchhoff polynomial $F_\Gamma$ is equal to any cofactor of its Laplacian $L_\Gamma$. 
In other words, for a graph $\Gamma=(V, E)$ with $|V|=r$, 
\[F_\Gamma=(-1)^{i+j}\det (L_{\Gamma}^{(ij)})\]
for any $1\leq i, j\leq r$. 
\end{theorem}

The following {Key Observation} \ref{rem:completegraph} is critically important {in the proof of} Theorem \ref{maintheorem2}.  {Actually, this observation builds a bridge between combinatorics of Kirchhoff polynomials and relative invariants of prehomogeneous vector space.} 

\begin{obs}\label{rem:completegraph}
{For a graph, we denote $x_{ij}=x_e$ for each edge $e=\{i, j\}$. 
In particular, $x_{ij}=x_{ji}$. For the complete graph $K_{r+1}$, 
the entries in Laplacian $L_{K_{r+1}}=\left(\ell_{ij}\right)_{1\leq i, j\leq r+1}$ is
\[\ell_{ij}=
\begin{cases}
\left(\sum_{k=1}^{r+1}x_{ik}\right)-x_{ii}&(\text{if \ $i=j$}), \\
-x_{ij}&(\text{otherwise}). 
\end{cases}\]
One can easily check that $L_{K_{r+1}}^{(11)}$ is a symmetric matrix and $\{x_{ij}\}_{1\leq i<j\leq r+1}$ gives a coordinate of the vector space $\Sym(r, \C)$ which consists of all $r\times r$ symmetric matrices over $\C$ (see below Example \ref{ex:laplacian} in the case $r=3$). 
Hence we have
\begin{align*}
\Set{L_{K_{r+1}}^{(11)}|x_{ij}\in\C}=\Sym(r, \C). 
\end{align*}
Therefore the Kirchhoff polynomial $F_{K_{r+1}}$ can be regarded as a function from $\Sym(r, \C)$ to $\C$. 
In other words, we can regard the Kirchhoff polynomial as the following function: 
\begin{align*}
F_{K_{r+1}}=\det: \Sym(r, \C)\to\C. 
\end{align*} }

\end{obs}
\begin{example}\label{ex:laplacian}
{The Laplacian matrix $L_{K_{4}}$ of the complete graph $K_{4}$ is
\begin{align*}
L_{K_{4}}=
\begin{pmatrix}
x_{12}+x_{13}+x_{14}&-x_{12}&-x_{13}&-x_{14} \\
-x_{21}&x_{21}+x_{23}+x_{24}&-x_{23}&-x_{24} \\
-x_{31}&-x_{32}&x_{31}+x_{32}+x_{34}&-x_{34} \\
-x_{41}&-x_{42}&-x_{43}&x_{41}+x_{42}+x_{43} 
\end{pmatrix}. 
\end{align*}
And the $(1,1)$ cofactor $L_{K_{4}}^{(11)}$ of $L_{K_{4}}$ is
\begin{align*}
L_{K_{4}}^{(11)}=
\begin{pmatrix}
x_{21}+x_{23}+x_{24}&-x_{23}&-x_{24} \\
-x_{32}&x_{31}+x_{32}+x_{34}&-x_{34} \\
-x_{42}&-x_{43}&x_{41}+x_{42}+x_{43} 
\end{pmatrix}. 
\end{align*}
Note that $L_{K_{4}}^{(11)}$ is a symmetric matrix and $\{x_{ij}\}_{1\leq i<j\leq r+1}$ gives a coordinate of $\Sym(3, \C)$. 
Hence we have
\begin{align*}
\Set{L_{K_{4}}^{(11)}|x_{ij}\in\C}=\Sym(3, \C). 
\end{align*}}

\end{example}

In \cite{MR0430336}, irreducible prehomogeneous vector spaces have already been classified. 
Here, we focus on the following prehomogeneous vector space whose the relative invariant is given by the Kirchhoff polynomial of complete graphs.  
See \cite[Proposition 3 in Section 5]{MR0430336} or \cite[Section 7, I-(2)]{MR0430336} for the details on Proposition \ref{KSSymtherelative}. 

\begin{proposition}\label{KSSymtherelative}
Let $\rho$ be the representation of $GL_r(\C)$ on $\Sym(r, \C)$ such that
\[\rho(P)X=PXP^T \ \ (P\in GL_r(\C)).\]
Then $(GL_r(\C), \rho, Sym(r, \C))$ is a regular irreducible prehomogeneous vector space. 
Moreover, the relative invariant is given by $\det: \Sym(r, \C)\to\C$.   
\end{proposition}

As stated in Key Observation \ref{rem:completegraph}, the Kirchhoff polynomial $F_{K_{r+1}}(\bm{x})$ of the complete graph $K_{r+1}$ is the relative invariant of the prehomogeneous vector space in Proposition \ref{KSSymtherelative}.

On the other hand, it is known that the evaluation of $(\det H_{F_{K_{r+1}}})|_{\bm{x}=(1,1,\ldots,1)}$ by the second author \cite{Y2018}. 
Note that we used Cayley's theorem $F_{K_{r+1}}(1, 1, \ldots, 1)=(r+1)^{r-1}$ at the  second equality in Proposition \ref{yazawa} (see \cite[Theorem VI. 30]{MR1813436} for details on Cayley's theorem).

\begin{proposition}[Yazawa {\cite[Theorem 3.3]{Y2018}}]\label{yazawa}\hspace{2pt}\\
For the complete graph $K_{r+1}$, 
\begin{align*}
(\det H_{F_{K_{r+1}}})|_{\bm{x}=(1,1,\ldots,1)}&={(-1)}^{N-1}{2}^{N-(r+1)}(r+1)^{r+1+N(r-3)}(r-1)\\
&=(-1)^{N-1}2^{N-r}(r-1)(F_{K_{r+1}}(1, 1, \ldots ,1))^{N-r-1}, 
\end{align*}
where $N=\binom{r+1}{2}$. 
\end{proposition}

By Corollary \ref{cor:reg} and Propositions \ref{KSSymtherelative}, \ref{yazawa}, we have Theorem \ref{maintheorem2}. 
\begin{remark}
Since Proposition \ref{yazawa} implies that $\det H_{F_{K_{}r+1}}\neq0$, one can see as Proposition \ref{yazawa} gives another proof of regularity of the prehomogeneous vector space in Proposition \ref{KSSymtherelative}. 
\end{remark}

\section{Applications}\label{sec:Applications}

In this section, 
we define a graded Artinian Gorenstein algebra $R_\Gamma$ associated to a graph $\Gamma$ (more generally, to a matroid), which has been previously introduced by Maeno and Numata in \cite{MR3566530}. 
Then by using strict (homogeneous) log-concavity of $F_\Gamma$ at any $\bm{a}\in(\R_{>0})^n$, 
we prove that $L_{\bm{a}}:=a_1x_1+\cdots+a_nx_n\in R_{F_\Gamma}^1$ satisfies the strong Lefschetz property at $R^1_{F_\Gamma}$. 
We also mention the relation between our result and known results by Huh and Wang in \cite{MR3733101} (Remark \ref{rem:HW}). 

\subsection{Artinian Gorenstein algebras}

First, we define an Artinian Gorenstein algebra associated to each homogeneous polynomial.  
\begin{definition}\label{def:algebra}
Let $F$ be a homogeneous polynomial of $F\in {\R}[x_1, \ldots, x_n]$. 
We define an ideal $\Ann(F)$ and a quotient algebra $R^*_F$ by
\begin{align*}
\Ann(F)&=\Set{P\in k[x_1, \ldots, x_n]  |  P\left(\fr<\partial/\partial x_1>, \ldots, \fr<\partial/\partial x_n>\right)F=0}, \\
R^*_F&=k[x_1, \ldots, x_n]/\Ann(F). 
\end{align*}
\end{definition}
\begin{definition}[Poincar\'{e} duality algebra cf.\ {\cite[Definition 2.1]{MR2594646}}]\hspace{2pt}\\
A finite-dimensional graded ${\R}$-algebra $R^*={\bigoplus_{\ell=0}^r{R^\ell}}$ is called the {\it Poincar\'{e} duality algebra} if $\dim_{{\R}} R^r=1$ and the bilinear pairing induced by the multiplication
\[R^\ell\times R^{r-\ell} \to R^r\]
is non-degenerate for $\ell=0, \ldots, \lfloor \fr<r/2> \rfloor$. 
\end{definition}

These rings $R^*_F$ can represent all (standard) graded Artinian Gorenstein algebras as the following. 

\begin{theorem}[cf.\ {\cite[Proposition 2.1, Theorem 2.1 and Remark 2.3]{MR2594646}}]\hspace{2pt}\\
Let $I$ be an homogeneous ideal of $k[x_1, \ldots, x_n]$ and $R^*:=k[x_1, \ldots, x_n]/I$ the quotient algebra, where $k$ is a field of characteristic zero. Then, the following are equivalent: 
\begin{itemize}
\item[(i)] The $k$-algebra $R^*$ is an Artinian Gorenstein algebra. 
\item[(ii)] There exists a homogeneous polynomial $F\in k[x_1, \ldots, x_n]$ such that $I=\Ann(F)$. 
\item[(iii)] $R^*$ is an Artinian Poincar\'{e} duality algebra.
\end{itemize}
\end{theorem}

We recall the notion of the strong Lefschetz property and the Hodge--Riemann bilinear relation.   

\begin{definition}[The strong Lefschetz property]\label{SLP}\hspace{2pt}\\
Let $R^*={\bigoplus_{\ell=0}^r{R^\ell}}, R^r\neq0$, be a graded Artinian ${\R}$-algebra. We say that $L\in R^1$ satisfies the {\it strong Lefschetz property} at degree $\ell$ (or ${R^\ell}$) if the multiplication map  
\[\times L^{r-2\ell} : R^\ell \to R^{r-\ell}\]
is bijective. 
\end{definition}
\begin{remark}
Our definition of the strong Lefschetz property is the strong Lefschetz property in the narrow sense in \cite[Definition 2.1]{MR2594646}. 
\end{remark}

We will use the following criterion which is the special case of the general criterion in \cite[Theorem 3.1]{MR2594646} and \cite[Theorem 4]{W2000}. 
\begin{theorem}[The Hessian criterion of the strong Lefschetz property cf.\ {\cite[Theorem 3.1]{MR2594646}}, {\cite[Theorem 4]{W2000}}]\label{The Hessian criterion}\hspace{2pt}\\
Assume that $x_1, \ldots, x_n\in R_F^1$ is a basis. 
An element $L_{\bm{a}}:=a_1x_1+\cdots+a_nx_n\in R_F^1$ satisfies the strong Lefschetz property at degree one if and only if $F(a_1, \ldots, a_n)\neq0$ and 
\[\det H_{F}|_{\bm{x}=\bm{a}}\neq0, \] 
where $H_F:=\left(\fr<\partial^2 F/\partial x_i\partial x_j>\right)_{1\leq i, j\leq n}$ is the Hessian matrix of $F$. 
\end{theorem}

\begin{definition}[Hodge--Riemann relation]\label{HRR}\hspace{2pt}\\
Let $R_F^*={\bigoplus_{\ell=0}^r{R_F^\ell}}$ be a graded Artinian Gorenstein $\R$-algebra associated to a homogeneous polynomial $F$ of degree $r$. We say that $L\in R_F^1$ satisfies the {\it Hodge--Riemann relation} at degree 1 (or $R_F^1$) 
if the Hodge--Riemann bilinear form $Q_{L}^{1}:R_F^{1}\times R_F^{1}\to {\R}, \ Q_{L}^{1}(\xi_1, \xi_2)={[\xi_1L^{r-2}\xi_2]}$ is negative definite on $\Ker(L^{r-1})$, 
where ${[-]} : R^r_F\xrightarrow{\sim} {\R}$ is {an isomorphism as $P \mapsto P\left(\fr<\partial/\partial x_1>, \ldots, \fr<\partial/\partial x_n>\right)F$}. 
\end{definition}

\begin{remark}
For a Poincar\'{e} duality algebra {$R_F^*$}, the Hodge--Riemann bilinear form ${{Q_{L_{\bm{a}}}^1}}$ at degree one is non-degenerate if and only if $L_{\bm{a}}$ satisfies the strong Lefschetz property at degree one. 
We also note that if $F(\bm{a})>0$, 
then $L_{\bm{a}}$ satisfies the Hodge--Riemann relation at $R^1_F$ if and only if $Q_{L_{\bm{a}}}^1$ is non-degenerate and has only one positive eigenvalue. 
In fact, first, note that $Q_{L_{\bm{a}}}^1(L_{\bm{a}}, L_{\bm{a}})=[L_{\bm{a}}^r]=r!F(\bm{a})>0$, where the final equality is deduced from the similar argument as below Remark \ref{rem:Hodge--Riemann}. Thus, the map $\times L_{\bm{a}}^{r} : R_F^0\xrightarrow{\times L_{\bm{a}}} R_F^1\xrightarrow{\times L_{\bm{a}}^{r-1}} R_F^r$ is an isomorphism, so $R_F^1=\R L_{\bm{a}}\oplus\Ker L_{\bm{a}}^{r-1}$. 
By definition, 
note that this decomposition is orthogonal with respect to $Q_{L_{\bm{a}}}^1$, and $Q_{L_{\bm{a}}}^1(L_{\bm{a}}, L_{\bm{a}})>0$. 
This implies that $L_{\bm{a}}$ satisfies the Hodge--Riemann relation at $R^1_F$ if and only if $Q_{L_{\bm{a}}}^1$ is non-degenerate and has only one positive eigenvalue.  
\end{remark}

\begin{remark}\label{rem:Hodge--Riemann}
Assume that $x_1, \ldots, x_n$ forms a basis of $R_F^1$. 
Then ${{Q_{L_{\bm{a}}}^1}}$ is given by the Hessian matrix $H_{F}|_{\bm{x}=\bm{a}}$ with respect to $x_1, \ldots, x_n\in R_F^1$. 
In fact, by definition, we have 
\begin{align*}
{Q_{L_{\bm{a}}}^1}(x_i, x_j)
&={[x_i L_{\bm{a}}^{r-2}x_j]} \\
&=\left(a_1\fr<\partial/\partial x_1>+\cdots+a_n\fr<\partial/\partial x_n>\right)^{r-2}\left(\fr<\partial^2/\partial x_i\partial x_j>\right)F \\
&={(r-2)!}\left.\fr<\partial^2 F/\partial x_i\partial x_j>\right|_{\bm{x}=\bm{a}}.
\end{align*}
(For more details, see the proof of \cite[Theorem 3.1]{MR2594646}.)
\end{remark}

\begin{remark}
In general, $x_1, \ldots, x_n$ is not necessarily linearly independent in $R^1_F$. 
For example, $F:=x_1x_2+x_1x_3+4x_1x_4+x_2x_3+x_2x_4+x_3x_4$ satisfies $-\fr<\partial F/\partial x_1>+2\fr<\partial F/\partial x_2>+2\fr<\partial F/\partial x_3>-\fr<\partial F/\partial x_4>=0$. 
In this case, $\det H_F$ is identically zero. 
\end{remark}

\subsection{The strong Lefschetz property of the Artinian Gorenstein algebras associated to simple graphic matroids}\label{sec:app1}
Here we consider the Artinian Gorenstein algebra $R^*_{F_M}$ associated to the basis generating function $F_M$ of a simple matroid $M$, 
particularly, the Kirchhoff polynomial $F_\Gamma$ of a simple graph $\Gamma$. 
In these cases, we will often abbreviate $R^*_{F_M}$ to $R^*_M$ and $R^*_{F_{\Gamma}}$ to $R^*_\Gamma$. 

Maeno and Numata conjectured that for any matroid $M$, the algebra $R^*_M$ has the strong Lefschetz property at all degrees in \cite{MR2985388}.

By our main result Theorem \ref{mainresult}, we have the following.

\begin{theorem}[The strong Lefschetz property of $R_\Gamma^*$ at degree one]\label{cor:SLPgraph}\hspace{2pt}\\
For any simple graph $\Gamma=(V, E)$ with $|V|=r+1\geq3$ and $|E|=n\geq3$, and any $\bm{a}=(a_1, \ldots, a_n)\in(\R_{>0})^n$, $a_1x_1+\cdots+a_nx_n\in R_{F_\Gamma}^1$ satisfies the strong Lefschetz property at degree one. 
\end{theorem}

\begin{theorem}[The {Hodge--Riemann relation of $R_\Gamma^*$ at degree one}]\label{cor:HRRgraph}\hspace{2pt}\\
In the above setting, for any $\bm{a}\in(\R_{>0})^n$, the Hodge--Riemann bilinear form ${Q_{L_{\bm{a}}}^1}$ is non-degenerate, {and ${Q_{L_{\bm{a}}}^1}$ has $n-1$ negative eigenvalues and one positive eigenvalue, i.e., the Hodge--Riemann relation holds.} 
\end{theorem}

Our result is not followed from related known results in \cite{MR3733101} {as the following}. 
\begin{remark}\label{rem:HW}
In \cite{MR3566530}, Maeno and Numata introduce an Artinian (generally non-Gorenstein) graded $k$-algebra $k[x_1, \ldots, x_n]/J_M$ associated to each matroid $M$, where $J_M$ is a certain ideal such that $J_M\subseteq\Ann F_M$. 
Then, they show that the strong Lefschetz property of $x_1+\cdots+x_n$ at every degree when $M=M(q, n)$ is the projective space over a finite field $\F_q$, in this case, $J_M=\Ann F_M$ (see \cite[Example 2.3 \& Theorem 4.3 (2)]{MR3566530}). 
In \cite{MR3733101}, Huh and Wang denote this ring by $B^*(M)={\bigoplus_{\ell=0}^r{B^{\ell}(M)}}$, and they study this ring associated to each general simple matroid. 
They show that this satisfies the ``injective'' Lefschetz property when $M$ is a representable matroid $M$, i.e., for any $0\leq \ell\leq \lfloor \fr<r/2>\rfloor$, 
the multiplication map $\times L^{r-2\ell} : B^\ell(M) \to B^{r-\ell}(M)$ of the element $L:=x_1+\cdots+x_n$ is injective, where $M$ is a representable matroid on $[n]$ of $\rank M=r$. 
Since we have the natural surjection $B^*(M)\twoheadrightarrow R^*_M$, we have the following commutative diagram: 
\[\begin{tikzcd}
B^1(M)\ar[d, twoheadrightarrow]\ar[r, hook, "{\times L}^{r-2}"]&B^{r-1}(M)\ar[d, twoheadrightarrow]\\
R^1_M\ar[r, "\times L^{r-2}"]&R^{r-1}_M
\end{tikzcd}\] 
This diagram would not imply our Corollary \ref{cor:SLPgraph}, that is, $\times L^{r-2} : R^1_M \to R^{r-1}_M$ is an isomorphism in the graphic case $M=M_\Gamma$. 
On the other hand, by the non-degeneracy of the Hessian of $F_{\Gamma}$, 
in particular, we know that $\{x_1, \ldots, x_n\}$ is a basis of $R^1_M$. 
Since $x_1, \ldots, x_n$ is a basis of $B^1(M)$ by the definition of $B^*(M)$, this implies that $B^1(M)=R^1_M$, and the Hodge--Riemann bilinear form at $B^1(M)$ and $R^1_M$ are the same. 
Thus Corollary \ref{cor:HRRgraph} implies the Hodge--Riemann bilinear form at $B^1(M_\Gamma)$ with respect to $L_{\bm{a}}$ is non-degenerate, moreover it has $n-1$ negative eigenvalues and one positive eigenvalue.   
\end{remark}

\subsection{The strong Lefschetz Property of elementary symmetric functions}\label{sec:app2}

Here we show that for any simple graph $\Gamma$ with $r$ vertices and $n$ edges, $1\leq \ell\leq r-2$, and $\bm{a}\in(\R_{>0})^n$, $(\partial_{\bm{a}})^\ell F_\Gamma$ is strictly homogeneous log-concave at $\bm{a}$, in particular strictly log-concave at $\bm{a}$. 
As an application, we prove that for elementary symmetric functions $e_{n-\ell}(x_1, \ldots, x_n)$ ($0\leq \ell\leq n-2$), $x_1+\cdots+x_n$ satisfies the strong Lefschetz property at degree one in $R_{e_{n-\ell}}^*$. 
  
First, we note the following general property of (strictly) homogeneous log-concave polynomial. 
In the proof, we use essentially the assumption of (strict) ``homogeneous'' log-concavity. 
Note that (strictly) log-concave does not imply Lemma \ref{lem:appa}.  

\begin{lemma}\label{lem:appa}
Let $F\in\R[x_1,\ldots,x_n]$ be a homogeneous polynomial of deg $F=r\geq3$. For any $\bm{a}\in\R^n$, the following are equivalent: 
\begin{itemize}
\item[(i)] $F$ is homogeneous log-concave (resp.\ strictly homogeneous log-concave) at $\bm{a}$, i.e., for any $s\geq\fr<r-1/r>$ (resp.\ $s>\fr<r-1/r>$),  
\[(-F H_F+s(\nabla F)^T(\nabla F))|_{\bm{x}=\bm{a}}\succeq 0 \ (\text{resp.\ } \ \succ 0).\]
\item[(ii)] $\partial_{\bm{a}}F:=\displaystyle\sum_{i=1}^n{a_i\fr<\partial F/\partial x_i>}$ is homogeneous log-concave (resp.\ strictly homogeneous log-concave) at $\bm{a}$, i.e., for any $s'\geq\fr<r-2/r-1>$ (resp.\ $s'>\fr<r-2/r-1>$),  
\[(-(\partial_{\bm{a}}F) H_{\partial_{\bm{a}}F}+s'(\nabla {\partial_{\bm{a}}F})^T(\nabla {\partial_{\bm{a}}F}))|_{\bm{x}=\bm{a}}\succeq 0 \ (\text{resp.\ } \ \succ 0).\] 
\end{itemize}
\end{lemma}
\begin{proof}
Assume that $\partial_{\bm{a}}F:=\sum_i{a_i\fr<\partial F/\partial x_i>}$ is strictly homogeneous log-concave at $\bm{a}$. By Euler's identity, we have the following identities. 
\[rF(\bm{a})=(\partial_{\bm{a}}F)(\bm{a}),\]
\[(r-1)\nabla F|_{\bm{x}=\bm{a}}=\nabla(\partial_{\bm{a}}F)|_{\bm{x}=\bm{a}},\]
\[(r-2)H_F|_{\bm{x}=\bm{a}}=H_{\partial_{\bm{a}}F}|_{\bm{x}=\bm{a}}.\]
Then, we have 
\begin{align*}
(-F H_F+s&(\nabla F)^T(\nabla F))|_{\bm{x}=\bm{a}} \\
&=\fr<1/r(r-2)>\left.\left\{-(\partial_{\bm{a}}F)H_{\partial_{\bm{a}}F}+s\cdot\fr<r(r-2)/(r-1)^2>(\nabla\partial_{\bm{a}}F)^T(\nabla\partial_{\bm{a}}F) \right\}\right|_{\bm{x}=\bm{a}}.
\end{align*}
If we set 
\begin{align*}
s'=s\cdot\fr<r(r-2)/(r-1)^2>, 
\end{align*}
then by some easy computations, we have
$s>\fr<r-1/r>$ if and only if $s'>\fr<r-2/r-1>$. 
This implies the equivalence of (i) and (ii). 
\end{proof}
\begin{corollary}
Let $F\in\R[x_1,\ldots,x_n]$ be a homogeneous polynomial of deg $F=r\geq3$. 
If $F$ is strictly homogeneous log-concave at $\bm{a}\in\R^n$, then for any $0\leq \ell\leq r-2$, $\partial_{\bm{a}}^\ell F:=(\partial_{\bm{a}})^\ell F$ is also strictly homogeneous log-concave at $\bm{a}\in\R^n$.
\end{corollary}

Then, by Corollary \ref{cor:eigen}, we have the following corollary. 

\begin{corollary}\label{cor:appb}
For any simple graph $\Gamma=(V, E)$ with $|V|=r+1\geq3$ and $|E|=n\geq3$, and any $\bm{a}\in(\R_{>0})^n$, 
$\partial_{\bm{a}}^\ell F_\Gamma$ is strictly homogeneous log-concave at $\bm{a}$, where $F_\Gamma$ is the Kirchhoff polynomial of $\Gamma$. 
In particular, $(-1)^{n-\ell-1}\det H_{\partial_{\bm{a}}^\ell F_\Gamma}|_{\bm{x}=\bm{a}}>0$, and $a_1x_1+\cdots+a_nx_n\in R_{\partial_{\bm{a}}^\ell F_\Gamma}^1$ satisfies the strong Lefschetz property at degree one.   
\end{corollary}

Let $e_{n-\ell}=e_{n-\ell}(x_1, \ldots, x_n)$ be the $(n-\ell)$-th elementary symmetric polynomial in $n$ variables. 
Then one can easily show the following identity:
\begin{equation*}
e_{n-\ell}(x_1, \ldots ,x_n)=\ell!\partial_{\bm{a}}^\ell e_n(x_1, \ldots ,x_n),
\end{equation*}
where $\bm{a}=(1,1, \ldots ,1)^T$. Since $e_n(x_1, \ldots, x_n)=x_1\cdots x_n$ is the Kirchhoff polynomial of a tree with $n+1$ vertices, we have the following by Corollary \ref{cor:appb}. 
\begin{corollary}
For the elementary symmetric polynomial $e_{n-\ell}=e_{n-\ell}(x_1, \ldots, x_n)$ $(0\leq \ell\leq n-2)$, the element $x_1+\cdots+x_n\in R^1_{e_{n-\ell}}$ satisfies the strong Lefschetz property at degree one. 
\end{corollary} 

\begin{remark}
In \cite[Theorem 4.3]{MN32012}, Maeno and Numata showed that for the $e_{k}(x_1, \ldots, x_n)$, 
the element $x_1+\cdots+x_n$ satisfies the strong Lefschetz property at all degrees. 
They used the Hessian criterion (see Theorem \ref{The Hessian criterion}) 
and showed the non-degeneracy of the Hessian matrix of $e_k(x_1, \ldots, x_n)$ at $(x_1, \ldots, x_n)=(1, \ldots, 1)$ by the non-degeneracy of the Poincar\'{e} duality of some Gorenstein algebra.   
\end{remark}

\bibliographystyle{amsplain-url}
\bibliography{mrbib}

\end{document}